\documentclass[11pt]{article}
\usepackage{amsmath,amssymb,amsthm}
\usepackage{amscd}
\newtheorem{theorem}{Theorem}[section]
\newtheorem{lemma}[theorem]{Lemma}
\newtheorem{proposition}[theorem]{Proposition}
\newtheorem{corollary}[theorem]{Corollary}

\theoremstyle{plain}

\theoremstyle{definition}
\newtheorem{definition}[theorem]{Definition}
\newtheorem{remark}[theorem]{Remark}

\numberwithin{equation}{section}

\renewcommand{\labelenumi}{\textup{(\theenumi)}}

\newcommand{\Homeo}{\operatorname{Homeo}}

\newcommand{\id}{\operatorname{id}}
\newcommand{\Ker}{\operatorname{Ker}}

\newcommand{\Ad}{\operatorname{Ad}}

\def\Re{{\operatorname{Re}}}
\def\det{{{\operatorname{det}}}}

\newcommand{\N}{\mathbb{N}}
\newcommand{\C}{\mathbb{C}}
\newcommand{\R}{\mathbb{R}}

\newcommand{\Z}{\mathbb{Z}}
\newcommand{\Zp}{{\mathbb{Z}}_+}

\title{On flow equivalence of one-sided topological Markov shifts}
\author{Kengo Matsumoto \\
Department of Mathematics \\
Joetsu University of Education \\
Joetsu, 943-8512, Japan
}
\date{}

\begin{document}
\maketitle

\def\det{{{\operatorname{det}}}}

\begin{abstract}
We introduce  notions of suspension and  flow equivalence on one-sided topological Markov shifts, which we call one-sided suspension and one-sided flow equivalence, respectively.
We prove that 
one-sided flow equivalence 
is equivalent to continuous orbit equivalence
on one-sided topological Markov shifts.
We also show that the zeta function of the flow 
on a one-sided suspension 
is a dynamical zeta function with some potential function
and that the set of certain dynamical zeta functions 
is invariant under one-sided flow equivalence of topological Markov shifts.
 \end{abstract}

2010{\it Mathematics Subject Classification}:
 Primary 37B10; Secondary  37C30.



\def\OA{{{\mathcal{O}}_A}}
\def\OB{{{\mathcal{O}}_B}}
\def\OTA{{{\mathcal{O}}_{\tilde{A}}}}
\def\SOA{{{\mathcal{O}}_A}\otimes{\mathcal{K}}}
\def\SOB{{{\mathcal{O}}_B}\otimes{\mathcal{K}}}
\def\SOTA{{{\mathcal{O}}_{\tilde{A}}\otimes{\mathcal{K}}}}
\def\FA{{{\mathcal{F}}_A}}
\def\FB{{{\mathcal{F}}_B}}
\def\DA{{{\mathcal{D}}_A}}
\def\DB{{{\mathcal{D}}_B}}
\def\DZ{{{\mathcal{D}}_Z}}
\def\DTA{{{\mathcal{D}}_{\tilde{A}}}}
\def\Ext{{{\operatorname{Ext}}}}
\def\Max{{{\operatorname{Max}}}}
\def\Per{{{\operatorname{Per}}}}
\def\PerB{{{\operatorname{PerB}}}}
\def\Homeo{{{\operatorname{Homeo}}}}
\def\HA{{{\frak H}_A}}
\def\HB{{{\frak H}_B}}
\def\HSA{{H_{\sigma_A}(X_A)}}
\def\Out{{{\operatorname{Out}}}}
\def\Aut{{{\operatorname{Aut}}}}
\def\Ad{{{\operatorname{Ad}}}}
\def\Inn{{{\operatorname{Inn}}}}
\def\det{{{\operatorname{det}}}}
\def\exp{{{\operatorname{exp}}}}
\def\cobdy{{{\operatorname{cobdy}}}}
\def\Ker{{{\operatorname{Ker}}}}
\def\ind{{{\operatorname{ind}}}}
\def\id{{{\operatorname{id}}}}
\def\supp{{{\operatorname{supp}}}}
\def\co{{{\operatorname{co}}}}
\def\Sco{{{\operatorname{Sco}}}}
\def\COE{{{\operatorname{COE}}}}


\section{Introduction}

Flow equivalence relation on two-sided topological Markov shifts is 
one of most imporatnt and interesting equivalence relations on symbolic 
dynamical systems.
It has a close relationship to
 classifications of not only continuous time dynamical systems but also 
 associated $C^*$-algebras.
For an irreducible square matrix $A = [A(i,j)]_{i,j=1}^N$ with its entries in $\{0,1\}$, the two-sided topological Markov shifts 
$(\bar{X}_A,\bar{\sigma}_A)$
is defined as a compact Hausdorff space   
$\bar{X}_A$ consistes of bi-infinite sequences 
$(\bar{x}_n)_{n\in \Z}$ 
of $\bar{x}_n \in \{1,2,\dots,N\}$
such that 
$A(\bar{x}_n,\bar{x}_{n+1}) =1, n\in \Z$ with shift homeomorphism
$\bar{\sigma}_A$ defined by
$
\bar{\sigma}_A((\bar{x}_n)_{n \in \Z}) 
=(\bar{x}_{n+1})_{n \in \Z}$.
For a positive continuous function 
$g$ on $\bar{X}_A$,
let us denote by 
$\bar{S}_A^g$
the compact Hausdorff space obtained from
\begin{equation*}
\{ (\bar{x}, r) \in  \bar{X}_A\times\R \mid \bar{x}\in\bar{X}_A,
0 \le r \le g(\bar{x}) \} 
\end{equation*}
by identifying
$(\bar{x}, g(\bar{x}))$
with
$(\bar{\sigma}_A(\bar{x}),0)$
for each $\bar{x} \in \bar{X}_A$.
Let
$
\bar{\phi}_{A,t}, t \in \R
$
be the flow on
$\bar{S}_A^g$
defined by 
$\bar{\phi}_{A,t}([(\bar{x},r)]) 
=[(\bar{x},r+t)]
$
for
$
[(\bar{x},r)] \in \bar{S}_A^g.
$
The dynamical system
$(\bar{S}_A^g, \bar{\phi}_A)$
is called the suspension of 
$(\bar{X}_A, \bar{\sigma}_A)$
by ceiling function $g$.
Two-sided topological Markov shifts
$(\bar{X}_A, \bar{\sigma}_A)$
and
$(\bar{X}_B, \bar{\sigma}_B)$
are said to be flow equivalent if 
there exists a positive continuous function $g$ on $\bar{X}_A$
such that 
$(\bar{S}_A^g, \bar{\phi}_A)$
is topologically conjugate to
$(\bar{S}_B^{1}, \bar{\phi}_B)$.
It is well-known that 
$(\bar{X}_A,\bar{\sigma}_A)$ 
and 
$(\bar{X}_B,\bar{\sigma}_B)$
are flow equivalent if and only if 
$\Z^N/(\id - A){\Z}^N$ 
is isomorphic to
$\Z^M/(\id - B){\Z}^M$
as abelian groups
and
$\det(\id - A) =\det(\id - B)$,
where $N, M$ are the sizes of the matrices $A, B$,
respectively
(\cite{BF}, \cite{Franks}, \cite{PS}).
Let us denote by
${\mathcal{K}}$ the $C^*$-algebra of compact operators on 
the separable infinite dimensional Hilbert space
$\ell^2(\N)$ and ${\mathcal{C}}$
its maximal abelian  $C^*$-subalgebra consisting of diagonal elements on
$\ell^2(\N)$.
Let us denote by $\OA$ the Cuntz--Krieger algebra and
by $\DA$ its canonical maximal abelian  $C^*$-subalgebra.
Since the group $\Z^N/(\id - A){\Z}^N$ 
is a complete invariant of the isomorphism class of 
the tensor product $C^*$-algebra
$\SOA$ (\cite{Ro}),
the $C^*$-algebra $\SOA$ with
$\det(\id - A)$
is a complete invariant for flow equivalence of 
the two-sided topological Markov shift
$(\bar{X}_A, \bar{\sigma}_A)$.
It has been recently shown in \cite{MMKyoto} that 
the isomorphism class of the pair
$(\OA\otimes{\mathcal{K}}, \DA\otimes{\mathcal{C}})$
is a complete invariant for flow equivalence class of 
$(\bar{X}_A, \bar{\sigma}_A)$.

One-sided topological Markov shifts $(X_A, \sigma_A)$
are also important and interesting class of dynamical systems.
The space $X_A$ 
 is defined as a compact Hausdorff space   
consisting of right infinite sequences 
$(x_n)_{n\in \N}$ 
of $x_n \in \{1,2,\dots,N\}$
such that 
$A(x_n,x_{n+1}) =1, n\in \N$ with continuous map
$\sigma_A$ defined by
$
\sigma_A((x_n)_{n \in \N}) 
=(x_{n+1})_{n \in \N}$.
In \cite{MaPacific},
the author has introduced a notion of continuous orbit equivalence
between one-sided topological Markov shifts.
 One-sided topological Markov shifts $(X_A, \sigma_A)$
 and
$(X_B,\sigma_B)$ are said to be continuously orbit equivalent
if there exists a homeomorphism 
$h: X_A \rightarrow X_B$ and continuous functions
$k_1,l_1: X_A \rightarrow \Zp$
and
$k_2,l_2: X_B \rightarrow \Zp$
such that 
\begin{align}
\sigma_B^{k_1(x)} (h(\sigma_A(x))) 
& = \sigma_B^{l_1(x)}(h(x))
\quad \text{ for } \quad 
x \in X_A,  \label{eq:orbiteqx} \\
\sigma_A^{k_2(y)} (h^{-1}(\sigma_B(y))) 
& = \sigma_A^{l_2(y)}(h^{-1}(y))
\quad \text{ for } \quad 
y \in X_B. \label{eq:orbiteqy}
\end{align}
In \cite{MMKyoto}, it has been proved 
that the isomorphism class of $\OA$ with
$\det(\id - A)$ is a complete invariant 
for continuous orbit equivalence  of 
the one-sided topological Markov shift
$(X_A, \sigma_A)$.
We have already known in  \cite{MaPacific}
 that 
the isomorphism class of the pair
$(\OA, \DA)$
is a complete invariant for continuous orbit equivalence  of 
$(X_A, \sigma_A)$.
Hence we may regard  
continuous orbit equivalence  of 
 one-sided topological Markov shifts
 as a one-sided counterpart of 
 flow equivalence of 
two-sided topological Markov shifts.

In this paper,
we will introduce a notion of flow equivalence on one-sided topological Markov shifts
$(X_A,\sigma_A)$.
We will first introduce a one-sided suspension $S_{A,b}^{l,k}$ 
with a flow $\phi_A$
associated to three real valued continuous functions 
$l,k,b\in C(X_A,\R)$
on $X_A$ for a one-sided topological Markov shift 
$(X_A,\sigma_A)$.
The space $S_{A,b}^{l,k}$ is determined by 
a base map $b: X_A \rightarrow \R$
and a ceiling function $l:X_A\rightarrow \R_+$
by identifying
$(x,r)$ with $(\sigma_A(x), r-(l-k))$ for
$r \ge l(x)$.
By using the one-sided suspension, 
we will define one-sided flow equivalence on
one-sided topological Markov shifts in Definition \ref{defn:onesidedfe}. 
As a main result of the paper, we will prove the following theorem.
\begin{theorem}[{Theorem \ref{thm:fecoe}}] \label{thm:thm1.1}
One-sided topological Markov shifts 
$(X_A, \sigma_A)$ and $(X_B,\sigma_B)$ 
are one-sided flow equivalent
if and only if they 
are continuously orbit equivalent.
\end{theorem}

By using \cite[Theorem 3.6]{MMETDS},
we see the following characterization of one-sided flow equivalence
which is a corollary of the above theorem.
\begin{corollary}[{Corollary \ref{cor:flmatrix}}]  
One-sided topological Markov shifts
$(X_A, \sigma_A)$ and $(X_B,\sigma_B)$ 
are one-sided flow equivalent
if and only if
there exists an isomorphism
$\varPhi: \Z^N/(\id - A){\Z}^N 
\rightarrow 
\Z^M/(\id - B){\Z}^M$
of abelian groups
such that 
$\varPhi([u_A]) = [u_B]$
and
$\det(\id - A) =\det(\id - B)$,
where
$[u_A]$ (resp. $[u_B]$)  
is the class of the vector 
$u_A =[1,\dots,1]$ in $\Z^N/(\id - A){\Z}^N$
(resp. 
$u_B =[1,\dots,1]$ in $\Z^M/(\id - B){\Z}^M$).
\end{corollary}

The zeta function $\zeta_{\phi}(s)$
of a flow $\phi_t:S \rightarrow S$ 
on a compact metric space $S$ with at most countably many closed orbits 
is defined by
\begin{equation}
\zeta_{\phi}(s) 
= \prod_{\tau \in P_{orb}(S,\phi)}
(1 - e^{-s \ell(\tau)})^{-1} 
\qquad
(\text{see } \cite{PP}, \cite{Ruelle1978}, \cite{Ruelle2002}, \text{ etc}.) 
\label{eqn:zetaflow} 
\end{equation}
where
$P_{orb}(S,\phi)$ denotes the set of primitive periodic orbits of
the flow $\phi_t:S\rightarrow S$
and 
$\ell(\tau)$ is the primitive length of the closed orbit 
defined by
 $\ell(\tau) = \min\{t\in \R_+ \mid \phi_t(u) = u\}$ 
for any point $u \in \tau$.
For a one-sided suspension $S_{A,b}^{l,k}$ 
of $(X_A,\sigma_A)$,
there exists a bijective correspondence between primitive periodic orbits
$\tau \in P_{orb}(S_{A,b}^{l,k},\phi_A)$ and periodic orbits 
$\gamma_\tau \in P_{orb}(X_A)$ of $(X_A,\sigma_A)$ 
such that
the length  
$
\ell(\tau)
$
of the orbit
$\tau$ is   
$\sum_{i=0}^{p-1}c(\sigma_A^i(x))
$
for $c = l-k$ and 
$\gamma_\tau =\{x,\sigma_A(x), \dots,\sigma_A^{p-1}(x)\}.
$
Thefore we have 
\begin{proposition}[{Proposition \ref{prop:zetaformula}}]
The zeta function 
$\zeta_{\phi_A}(s)$
of the flow $\phi_A$
of the one-sided suspension 
$S_{A,b}^{l,k}$
of $(X_A,\sigma_A)$
is given by the dynamical zeta function 
$\zeta_{A,c}(s)$ with potential function 
$c=l-k$ such that 
\begin{equation}
\zeta_{A,c}(s)
= \exp\{
\sum_{n=1}^\infty\frac{1}{n}\sum_{{x}\in \Per_n({X}_A)}
\exp( -s \sum_{i=0}^{n-1}{c}({\sigma}_A^i({x}))) \}. 
\label{eq:dynamiczeta}
\end{equation}
\end{proposition}

In \cite{BH}, Boyle--Handelman have studied 
a relationship between flow equivalence of two-sided topological Markov shifts
and dynamical zeta functions on them,
so that they 
have proved that
the set of zeta functions of homeomorphisms flow equivalent to 
the two-sided topological Markov shift
$(\bar{X}_A, \bar{\sigma}_A)$
is a complete invariant for flow equivalence class of 
$(\bar{X}_A, \bar{\sigma}_A)$.
If  $(X_A, \sigma_A)$ and $(X_B,\sigma_B)$ 
are continuously orbit equivalent
via a homeomorphism $h:X_A\rightarrow X_B$
with continuous functions
$k_1,l_1:X_A\rightarrow \Zp$ and
$k_2,l_2:X_B\rightarrow \Zp$ 
satisfying \eqref{eq:orbiteqx} and \eqref{eq:orbiteqy},
we may define  homomorphisms
$\Psi_h: C(X_B,\Z) \rightarrow C(X_A,\Z)$
by
\begin{equation}
\Psi_{h}(f)(x)
= \sum_{i=0}^{l_1(x)-1} f(\sigma_B^i(h(x))) 
- \sum_{j=0}^{k_1(x)-1} f(\sigma_B^j(h(\sigma_A(x))))
\label{eq:Psihfx}
\end{equation}
for
$f \in C(X_B,\Z), \ x \in X_A$ 
and similarly 
$\Psi_{h^{-1}}: C(X_A,\Z) \rightarrow C(X_B,\Z)$
for $h^{-1}: X_B \rightarrow X_A$.
In \cite{MMETDS},
it has been proved that 
$\Psi_h: C(X_B,\Z) \rightarrow C(X_A,\Z)$
induces an isomorphism of their ordered cohomology groups
$(H^B,H^B_+)$ 
and
$(H^A,H^A_+)$.
If $(X_A,\sigma_A)$ and $(X_B,\sigma_B)$
are one-sided flow equivalent,
they are continuously orbit equivalent, so that we may define 
the map $\Psi_h: C(X_B,\Z) \rightarrow C(X_A,\Z)$
as above.
Inspired by  \cite{BH},
we will show the following.
\begin{theorem}[{Theorem \ref{thm:fezeta}}]
Suppose that
$(X_A, \sigma_A)$ and $(X_B,\sigma_B)$
are one-sided flow equivalent 
via a homeomorphism
$h:X_A\rightarrow X_B$.
Then  for $f \in C(X_B, \Z), g \in C(X_A, \Z)$
such that
the classes $[f], [g]$ are order units of the ordered cohomology groups
$(H^B,H^B_+),(H^A,H^A_+)$, respectively.
Then 
we have
\begin{equation*}
 \zeta_{A,\Psi_h(f)}(s)  = \zeta_{B,f}(s),
 \qquad
 \zeta_{B,\Psi_{h^{-1}}(g)}(s)  = \zeta_{A,g}(s).
\end{equation*}
\end{theorem}
This theorem shows  that 
the set $Z(X_A,\sigma_A)$
of dynamical zeta functions of $(X_A,\sigma_A)$
whose potential functions are order units in the ordered cohomology
group $(H^A,H^A_+)$ is invariant under one-sided flow equivalence.
\medskip


Throughout the paper, 
we denote 
by $\R_+$, by $\Zp$ and 
by $\N$
the set of nonnegative real numbers,
the set of nonnegative integers and the set of positive integers, 
respectively.

\section{One-sided suspensions}
In what follows, we assume that 
$A =[A(i,j)]_{i,j=1}^N$
is an $N \times N$ matrix with entries in $\{0,1\}$
and $1<N \in \N$.
We further assume that $A$ is irreducible and not any permutation matrix.
This assumption is equivalent to the condition (I) in the sense of \cite{CK}
so that the space $X_A$ is homeomorphic to a Cantor discontinuum.  
We denote by $C(X_A,\R)$ (resp. $C(X_A,\R_+)$)
the set of real  (resp. nonnegative real)  
valued continuous functions on $X_A$.
The set $C(X_A,\Z)$ 
of integer valued continuous functions on $X_A$
has a natural structure of abelian group 
by pointwise sums.
Let us denote by $H^A$ the quotient group
of the abelian group $C(X_A, \Z)$ by the subgroup 
$\{ g - g \circ \sigma_A \mid g \in C(X_A,\Z) \}$.
The positive cone $H^A_+$ consists of the classes
$[f]\in H^A$ of nonnegative integer valued continuous functions 
$f \in C(X_A,\Zp)$.
The ordered group
$(H^A,H^A_+)$ is called the ordered cohomology group for 
$(X_A,\sigma_A)$ (cf. \cite{BH}, \cite{MMKyoto}, \cite{Po}).
For $f \in C(X_A,\Z)$ and $m \in \N$, we set 
$$
f^m(x) =\sum_{i=0}^{m-1}f(\sigma_A^i(x)), \qquad x \in X_A.
$$
An element $[f]$ in $H^A_+$ is  
called an order unit if for any $[g] \in H^A$,
there exists $n \in \N$ such that 
$n[f] -[g] \in H^A_+$.
We see that 
$[f]\in H^A_+$ is an order unit if and only if 
there exists $m \in \N$ such that 
$f^m$ is strictly positive (\cite[1,7]{BH}).

In this section,
we will first define a one-sided suspension
for a one-sided topological Markov shift.
The triplet
$(l,k,b)$ for real valued continuous functions
$l,k \in C(X_A,\R_+)$ and $b \in C(X_A,\R)$
are said to be {\it suspension triplet}\/
for $(X_A,\sigma_A)$
if they satisfy the following two conditons:
\begin{enumerate}
\item The difference $c = l-k$ belongs to $C(X_A,\Z)$
and the class $[c]\in H_+^A$ is an order unit of $(H^A,H_+^A)$.
\item The differences $l-b, \, k - b\circ \sigma_A$  
belong to $C(X_A,\Zp)$. 
\end{enumerate}
The triplet $(1,0,0)$ is called the standard suspension triplet.

We fix a suspension triplet $(l,k,b)$ for $(X_A,\sigma_A)$
for a while.
We set $X_{A,b}^{\R} = \{(x,r) \in X_A\times \R \mid r \ge b(x)\}$
and 
define an equivalence relation
$\underset{l,k}{\sim}$ in 
$X_{A,b}^{\R}$ generated by the relations
\begin{equation}
(x,r) \underset{l,k}{\sim} (\sigma_A(x), r - c(x)) \qquad 
\text{ for } r \ge l(x).
\end{equation}
We note that the condition $r\ge l(x)$ implies $(x,r) \in X_{A,b}^{\R}$
and
$r-c(x) = r-l(x) +k(x) \ge k(x) \ge b(\sigma_A(x))$
so that 
$(\sigma_A(x), r - c(x)) \in X_{A,b}^{\R}$. 
If there exists $n \in \Zp$ such that 
$r \ge c^m(x) + l(\sigma_A^m(x))$ 
for all $m\in \Zp$ with $0\le m < n$,
then 
\begin{equation*}
(x,r) \underset{l,k}{\sim} 
(\sigma_A(x), r - c(x)) \underset{l,k}{\sim} \cdots
\underset{l,k}{\sim} 
(\sigma_A^n(x), r - c^n(x)).
\end{equation*} 
Hence we have
\begin{lemma}\label{lem:equiviff}
For $(x,r), (x',r') \in X_{A,b}^{\R},$
we have
$(x,r)
\underset{l,k}{\sim}
 (x',r')
$
if and only if
there exist $n, n' \in \Zp$ such that 
\begin{gather*}
\sigma_A^n(x)=\sigma_A^{n'}(x'),
\quad
r -c^n(x) =r' -c^{n'}(x') \quad{ and}\\
r - c^m(x) \ge l(\sigma_A^m(x)),
\quad  
r - c^{m'}(x') \ge l(\sigma_A^{m'}(x')) 
\quad \text{ for } 
\quad
0\le m <n, \,
0\le m' <n'.  
\end{gather*}
\end{lemma}
\begin{proof}
It suffices to prove the only if part.
For $(x,r) \in X_{A,b}^{\R}$
with $r \ge l(x)$,
we write a directed edge from
$(x,r)$ to 
$(\sigma_A(x),r-c(x))$.
We then have a directed graph with vertex set 
$X_{A,b}^{\R}$.
Suppose 
$(x,r)
\underset{l,k}{\sim}
 (x',r').
$
There exists a finite sequence
$(x_i,r_i) \in X_{A,b}^{\R}, i=0,1,\dots, L$
such that 
$(x_0,r_0) = (x,r)$
and
$(x_L,r_L) = (x',r')$,
and
there exists a directed edge 
from
$(x_{i-1}, r_{i-1})$ to $(x_i,r_i)$
or
from
$(x_i, r_i)$ to $(x_{i-1},r_{i-1})$
for each $i=1,2,\dots,L$.
Since each vertex $(x_i,r_i)$ 
emits  at most  one directed edge,
we may find $n$ with $0\le n \le L$
such that 
there exist directed edges 
from
$(x_{i-1}, r_{i-1})$ to $(x_i,r_i)$
for $i=1,2,\dots, n$
and
from
$(x_i, r_i)$ to $(x_{i-1},r_{i-1})$
for $i=n+1,\dots,L$.
By putting $n' = L-n$,
we see that $n$ and $n'$ 
satisfy the desired conditions for $(x,r)$ and $(x',r')$.
\end{proof}
Define a topological space
\begin{equation}
S_{A,b}^{l,k} = X_{A,b}^{\R}/ \underset{l,k}{\sim}
\end{equation}
as the quotient topological space of $X_{A,b}^{\R}$ 
by the equivalence relation $\underset{l,k}{\sim}$.
We denote by $[x,r]$ the class of 
$(x,r) \in X_{A,b}^{\R}$ in the quotient space 
$
S_{A,b}^{l,k}.
$
We will show that 
$
S_{A,b}^{l,k} 
$
is a compact Hausdorff space.
We note the following lemma.
\begin{lemma}\label{lem:orderunitpos}
For any $m \in\N$, there exists $n_m\in \N$ such that
$c^{n_m}(x) \ge m$ for all $x \in X_A$. 
\end{lemma}
\begin{proof}
Since $[c] \in H^A_+$ is an order unit,
one may take $p\in \N$ such that 
$c^p$ is a strictly positive function
so that 
$c^p(x) \ge 1$ for all
$x \in X_A$.
By the identity
$c^{mp}(x) = c^{(m-1)p}(x) + c^p(\sigma_A^{(m-1)p}(x))
$ for $m\in \N, x \in X_A$,
one obtains that
$c^{mp}(x) \ge m$ for all $x \in X_A$.
By putting $n_m = mp$, we see the desired assertion.
\end{proof}
We set
\begin{align*}
\Omega_{A,b}^l & = \{(x,r) \in X_{A,b}^{\R}\mid b(x) \le r \le l(x) \},\\
{\Omega}_{A,b}^{l\circ} 
& = \{(x,r) \in X_{A,b}^{\R}\mid b(x) \le r < l(x) \}.
\end{align*}
\begin{lemma}\label{lem:Omegacirc}
For $(x,r) \in X_{A,b}^{\R}$
with $r \ge l(x)$,
there exists
$(z,s) \in {\Omega}_{A,b}^{l\circ}$
such that 
$
(x,r) \underset{l,k}{\sim} (z,s).
$
\end{lemma}
\begin{proof}
For $(x,r) \in X_{A,b}^{\R}$
with $r \ge l(x)$,
by Lemma \ref{lem:orderunitpos}
one may take
 a minimum number $n \in \Zp$
 satisfying 
$r < c^{n+1}(x) + k(\sigma_A^n(x))$,
so that we have
\begin{equation*}
  c^m(x) +k(\sigma_A^{m-1}(x)) 
\le r 
< c^{n+1}(x) + k(\sigma_A^n(x)) 
\quad\text{ for  all }
\quad
m \le n.
\end{equation*}
In particular, we have
\begin{equation*}
  c^n(x) +k(\sigma_A^{n-1}(x)) 
\le r 
< c^{n+1}(x) + k(\sigma_A^n(x)) .
\end{equation*}
As 
$c^{n+1}(x) = c^n(x) + l(\sigma_A^n(x))-k(\sigma_A^n(x))$
and 
$b(\sigma_A^n(x)) \le k(\sigma_A^{n-1}(x))$,
we get
\begin{equation*}
b(\sigma_A^n(x)) \le  r - c^n(x) < l(\sigma_A^n(x)).
\end{equation*}
Since
$
 r - c^m(x) \ge l(\sigma_A^m(x)) 
$
for all
$ m \in \Zp $
with
$m < n,
$
we see that 
$$
(x,r) \underset{l,k}{\sim}
(\sigma_A^m(x),  r - c^m(x))
\quad
 \text{ for all } m \in \Zp 
 \text{ with }
 m \le n.
$$
By putting
$z = \sigma_A^n(x), s = r - c^n(x)$, 
we have
$
(x,r) \underset{l,k}{\sim} (z,s)
$
and
$(z,s) \in {\Omega}_{A,b}^{l\circ}$.
\end{proof}
\begin{proposition}
$
S_{A,b}^{l,k} 
$
is a compact Hausdorff space.
\end{proposition}
\begin{proof}
We may restrict the equivalence relation
$\underset{l,k}{\sim}$
to
${\Omega}_{A,b}^{l}$.
Let
$q_\Omega:
{\Omega}_{A,b}^{l}\rightarrow {\Omega}_{A,b}^{l}/\underset{l,k}{\sim}
$
and
$q_X:
{X}_{A,b}^{\R}\rightarrow {S}_{A,b}^{l,k}
$
be the quotient maps, which are continuous.
Since
${\Omega}_{A,b}^{l}$
is compact, so is 
${\Omega}_{A,b}^{l}/\underset{l,k}{\sim}.
$
By Lemma \ref{lem:Omegacirc},
an element of 
${\Omega}_{A,b}^{l}/\underset{l,k}{\sim}
$
is represented by
${\Omega}_{A,b}^{l\circ}/\underset{l,k}{\sim}$.
This implies that
${\Omega}_{A,b}^{l}/\underset{l,k}{\sim}
$
is Hausdorff,
because ${\Omega}_{A,b}^{l\circ}$ 
is a fundamental domain
of the quotient space
${\Omega}_{A,b}^{l}/\underset{l,k}{\sim}$.
For $(x,r) \in X_{A,b}^\R$,
take a minimum number $n \in \Zp$
such that 
$r < c^{n+1}(x) + k(\sigma_A^n(x))$,
and define a continuous map
$\varphi:X_{A,b}^\R \rightarrow \Omega_{A,b}^l$
by setting
$\varphi((x,r)) = (\sigma_A^n(x), r - c^n(x))$.
By the proof of Lemma \ref{lem:Omegacirc},
it induces a map
$\tilde{\varphi}:S_{A,b}^{l,k}\rightarrow 
{\Omega}_{A,b}^{l}/\underset{l,k}{\sim}.
$
We will see that $\tilde{\varphi}$ is a homeomorphism,
so that 
$S_{A,b}^{l,k}$ is a compact Hausdorff space.
As the inclusion map 
$\iota:(x,r) \in {\Omega}_{A,b}^{l} \rightarrow (x,r) \in X_{A,b}^\R$
induces a map
$$
\tilde{\iota}: 
[x,r] \in {\Omega}_{A,b}^{l}/\underset{l,k}{\sim}
\quad \rightarrow 
\quad 
[x,r] \in S_{A,b}^{l,k}
$$
which satisfies
$\tilde{\varphi}\circ\tilde{\iota} =\id, \, 
\tilde{\iota}\circ\tilde{\varphi} =\id$,
the map
$\tilde{\varphi}$ is bijective.
We have commutative diagrams:
\begin{equation*}
{
\begin{CD}
X_{A,b}^\R @>\varphi>> {\Omega}_{A,b}^{l} \\
@V{q_X}VV  @VV{q_\Omega}V \\
S_{A,b}^{l,k} @>\tilde{\varphi}>> 
{\Omega}_{A,b}^{l}/\underset{l,k}{\sim} 
\end{CD},
}
\qquad\quad
{
\begin{CD}
X_{A,b}^\R @<\iota<< {\Omega}_{A,b}^{l} \\
@V{q_X}VV  @VV{q_\Omega}V \\
S_{A,b}^{l,k} @<\tilde{\iota}<< 
{\Omega}_{A,b}^{l}/\underset{l,k}{\sim} 
\end{CD}
}
\end{equation*}
Since both the maps
$\varphi: X_{A,b}^\R \rightarrow {\Omega}_{A,b}^{l}$
and
$\iota: {\Omega}_{A,b}^{l} \rightarrow X_{A,b}^\R$
are continuous,
the commutativity of the diagrams 
imply the continuity of the maps
$\tilde{\varphi}: S_{A,b}^{l,k} 
\rightarrow \Omega_{A,b}^l/\underset{l,k}{\sim}$
and
$\tilde{\iota}: \Omega_{A,b}^l/\underset{l,k}{\sim}
\rightarrow S_{A,b}^{l,k}$.
Hence
$\tilde{\varphi}: S_{A,b}^{l,k} 
\rightarrow \Omega_{A,b}^l/\underset{l,k}{\sim}$
is a homeomorphism.
\end{proof}

We will define  the flow 
$\phi_{A,t}, t \in \R_+$ on $S_{A,b}^{l,k}$ 
by 
$\phi_{A,t}([x,r]) = [x,r+t]$ for $t \in \R_+$.
As in the discussions above,
for $(x,r) \in X_{A,b}^{\R}$ and $t \in \R_+$,
there exists $n \in \Zp$ such that 
\begin{equation*}
b(\sigma_A^n(x)) \le t + r - c^n(x) < l(\sigma_A^n(x))
\end{equation*}
and
\begin{equation*}
\phi_{A,t}([x,r]) = [\sigma_A^n(x), t + r - c^n(x)].
\end{equation*}
Hence the flow $\phi_{A,t}([x,r]),t \in \R_+$
are defined in $S_{A,b}^{l,k}$.
We call the flow space 
$(S_{A,b}^{l,k}, \phi_{A})$
the $(l,k,b)$-suspension of one-sided topological Markov shift
$(X_A, \sigma_A)$.
It is simply called the {\it one-sided suspension}\/
of $(X_A, \sigma_A)$.
The map
$b_A: X_A \rightarrow S_{A,b}^{l,k}$
defined by
$b_A(x) = [x,b(x)]$
is called the base map. 
The base map
for $b\equiv 0$
is written
$s_A(x) =[x,0]$
and called 
the standard base map.
If all the functions
$l,k, b$ are  valued in integers, 
$(l,k,b)$-suspension is called 
$(l,k,b)$-discrete suspension.
For $l\equiv 1, k\equiv 0, b\equiv 0$, 
the $(1,0,0)$-suspension 
$(S_{A,0}^{1,0}, \phi_A)$
is called the standard one-sided suspension.
The $(l,k,0)$-suspension space $S_{A,0}^{l,k}$
and the $(l,0,0)$-suspension space $S_{A,0}^{l,0}$
are denoted by
$S_{A}^{l,k}$ and $S_{A}^{l}$, respectively.
Hence the standard one-sided suspension
$(S_{A,0}^{1,0}, \phi_A)$
is denoted by
$(S_{A}^{1}, \phi_A)$.
\begin{lemma}
The standard base map
$s_A: X_A \rightarrow S_{A}^{1}$
for the standard suspension
is an injective homeomorphism.
\end{lemma}
\begin{proof}
It suffices to show the injectivity of the map $s_A$.
Suppose that
$s_A(x) = s_A(z)$ in $S_{A}^{1}$
for some $x,z \in X_A$.
There exists a finite sequence
$(x_i,r_i) \in X_{A,0}^{\R}$
such that
$$
(x,0) \underset{1,0}{\sim} 
(x_1,r_1) \underset{1,0}{\sim} \cdots \underset{1,0}{\sim}
(x_n,r_n) \underset{1,0}{\sim} 
(z,0)
$$ 
where
$r_i\in \Zp, i=1,\dots,n.$
If $r_i = r_{i+1}$, we may take $x_i = x_{i+1}$.
Let
$K = \Max\{r_i \mid i=1,\dots,n\}$.
Take $i_K \in \{1,\dots,n\}$ such that
$r_{i_K} = K$.
It then follows that
$\sigma_A^K(x_{i_K}) = x$ and
$\sigma_A^K(x_{i_K}) = z$,
so that $x = z$.
\end{proof}
\begin{lemma}
For a suspension triplet
$(l,k,b)$ for $(X_A,\sigma_A)$,
put $c = l-k$.
If $c' \in C(X_A,\Z)$  satisfies $[c] = [c'] $ in $H^A$,
there exists a suspension triplet
$(l',k',b')$ for $(X_A,\sigma_A)$ 
such that
$c' = l'-k'$
and
there exists a homeomorphism
$\Phi: S_{A,b}^{l,k} \rightarrow S_{A,b'}^{l',k'}
$
satisfying
\begin{equation}
\Phi \circ b = b', \qquad \Phi \circ \phi_{A,t} = \phi_{A,t}\circ \Phi\quad
\text{ for } t \in \R_+. \label{eq:2.6}
\end{equation}
Hence  
$(S_{A,b}^{l,k}, \phi_A)$ and $( S_{A,b'}^{l',k'}, \phi_A)$
are topologically conjugate
compartible to their base maps.
\end{lemma}
\begin{proof}
Since $[c] = [c']$ in $H^A$,
there exists $d \in C(X_A, \Z)$ such that 
$c - c' = d\circ\sigma_A - d$.
We may assume that 
$d(x) \in \Zp$ for all $x \in X_A$.
Define
\begin{equation*}
l'(x) = l(x) +d(x), \qquad 
k'(x) = k(x) +d(\sigma_A(x)), \qquad  
b'(x) = b(x) +d(x), \qquad x \in X_A. 
\end{equation*}
It is easy to see that
$c' = l' -k'$
and
$(l',k',b')$ is a suspension triplet for $(X_A,\sigma_A)$.
Define
$\Phi:X_{A,b}^{\R} \rightarrow X_{A,b'}^{\R}$
by
$\Phi((x,r)) = (x,r+d(x))$.
We know that 
$(x,r) \in X_{A,b}^{\R}$
if and only if
$(x, r+d(x))\in X_{A,b'}^{\R}$.
Since
$\Phi((\sigma_A(x),r-c(x))) 
= (\sigma_A(x),r-c(x) + d(\sigma_A(x)))
= (\sigma_A(x),r +d(x) -c'(x)),
$
we have
$\Phi((x,r)) \underset{l',k'}{\sim}
\Phi((\sigma_A(x),r-c(x)))$
for $r \ge l'(x)$. 
It is easy to see that 
$\Phi$ extends to a homeomorphism
$S_{A,b}^{l,k} \rightarrow S_{A,b'}^{l',k'}
$
which is still denoted by $\Phi$
satisfying the equalities \eqref{eq:2.6}.
\end{proof}

\section{One-sided flow equivalence}
We will define one-sided flow equivalence 
on one-sided topological Markov shifts.
\begin{definition}\label{defn:onesidedfe}
$(X_A,\sigma_A)$ and 
$(X_B,\sigma_B)$ are said to be {\it one-sided flow equivalent}\/
if there exist suspension triplets $(l_1,k_1,b_1)$ for $(X_A,\sigma_A)$
and $(l_2,k_2,b_2)$ for  $(X_B,\sigma_B)$,
a homeomorphism
$h:X_A \rightarrow X_B$,
and continuous maps
$\Phi_1:S_{A,b_1}^{l_1,k_1}\rightarrow S_{B}^{1}$,
$\Phi_2:S_{B,b_2}^{l_2,k_2}\rightarrow S_{A}^{1}$
such that
\begin{alignat}{3}
\Phi_1\circ \phi_{A,t} & = \phi_{B,t}\circ \Phi_1
\quad\text{for } t \in \R_+, \qquad
\Phi_1\circ b_{1,A}&=& s_B \circ h, \label{eq:crosssectionB}
\\
\Phi_2\circ \phi_{B,t} & = \phi_{A,t} \circ \Phi_2
\quad\text{for } t \in \R_+,
\qquad
\Phi_2\circ b_{2,B}&=& s_A \circ h^{-1},\label{eq:crosssectionA}
\end{alignat}
where 
$b_{1,A}:X_A\rightarrow S_{A,b_1}^{l_1,k_1}$
and
$b_{2,B}:X_B\rightarrow S_{B,b_2}^{l_2,k_2}$
are the base maps defined by
$b_{1,A}(x) = [x,b_1(x)]$ for $x \in X_A$ 
and 
$b_{2,B}(y) = [y,b_2(y)]$ for $y \in X_B$,
respectively. 
\end{definition}
In this case, we say that 
$(X_A,\sigma_A)$ and 
$(X_B,\sigma_B)$ are  one-sided flow equivalent
via a homeomorphism $h:X_A\rightarrow X_B$.
If there exists a homeomorphism $h:X_A\rightarrow X_A$
satisfying \eqref{eq:crosssectionB} (resp. \eqref{eq:crosssectionA}),
we say that $(X_B,\sigma_B)$ (resp. $(X_A,\sigma_A)$)
is a {\it cross section}\/ of the one-sided suspension
$S_{A,b_1}^{l_1,k_1}$ through $b_{1,A}\circ h^{-1}$
(resp.
$S_{B,b_2}^{l_2,k_2}$ through $b_{2,B}\circ h $).
\begin{proposition}\label{prop:3.2}
If 
$(X_A,\sigma_A)$ and 
$(X_B,\sigma_B)$ are continuously orbit equivalent,
then they are one-sided flow equivalent.
\end{proposition}
\begin{proof}
Let $h:X_A\rightarrow X_B$
be a homeomorphism which gives rise to a continuous orbit equivalence
between
$(X_A,\sigma_A)$ and 
$(X_B,\sigma_B)$ with 
continuous functions 
$k_1,l_1: X_A\rightarrow\Zp$ 
and 
$k_2,l_2:X_B\rightarrow \Zp$
satisfying \eqref{eq:orbiteqx}
and \eqref{eq:orbiteqy}, respectively.
Put
$c_1(x) = l_1(x) -k_1(x), x \in X_A$
and
$c_2(y) = l_2(y) -k_2(y), y \in X_B$.
We set $b_1\equiv 0$, $b_2\equiv 0$.
By \cite[Theorem 5.11]{MMETDS},
the map
$\Psi_h: C(X_B,\Z) \rightarrow C(X_A,\Z)$
defined by \eqref{eq:Psihfx}
induces an isomorphism of ordered
groups between
$(H^B, H^B_+)$ and $(H^A, H^A_+)$,
so that the elements $\Psi_h(1)= c_1$ and similarly
$\Psi_{h^{-1}}(1) =c_2$
give rise to an order unit of $(H^A, H^A_+)$
and
of $(H^B, H^B_+)$, respectively.
Then 
$(l_1,k_1,b_1)$ is a suspension triplet for $(X_A,\sigma_A)$
and
$(l_2,k_2,b_2)$ is a suspension triplet for $(X_B,\sigma_B)$.
Define 
$\Phi_1:X_{A,0}^{\R}\rightarrow X_{B,0}^{\R}$
by
$\Phi_1((x,r)) = (h(x),r)$ for $x \in X_A, r\ge 0$.
As 
$S_{B}^{1}$
is the standard suspension,
we have for $(x,r) \in X_{A,0}^{\R}$ with $r \ge l_1(x)$
\begin{equation*}
(h(x),r) \underset{1,0}{\sim}(\sigma_B^{l_1(x)}(h(x)), r-l_1(x))
\end{equation*}
and
\begin{equation*}
(h(\sigma_A(x)),r-c_1(x)) 
\underset{1,0}{\sim}
(\sigma_B^{k_1(x)}(h(\sigma_A(x))), r-c_1(x)-k_1(x)).
\end{equation*}
Since
$\sigma_B^{l_1(x)}(h(x))=\sigma_B^{k_1(x)}(h(\sigma_A(x)))$
and
$ r-l_1(x) = r-c_1(x)-k_1(x)$,
we have
\begin{equation*}
(h(x),r) \underset{1,0}{\sim} (h(\sigma_A(x)),r-c_1(x))
\quad \text{ in } S_{B}^{1}
\end{equation*}
so that the map
$\Phi_1:X_{A,0}^{\R}\rightarrow X_{B,0}^{\R}$
induces a continuous map
$S_{A}^{l_1,k_1}\rightarrow S_{B}^{1}$
which is still denoted by 
$\Phi_1$.
It is clear to see that
the equalities
$\Phi_1\circ \phi_{A,t} = \phi_{B,t}\circ \Phi_1$
for $ t \in \R_+$
and
$\Phi_1\circ b_{1,A} = s_B\circ h$
hold.
We similarly have a continuous map
$\Phi_2:S_{B}^{l_2,k_2}\rightarrow S_{A}^{1}$
defined by
$\Phi_2([y,s]) = [h^{-1}(y),s]$ satisfying
the equalities
$\Phi_2\circ \phi_{B,t} = \phi_{A,t}\circ \Phi_2$
for $ t \in \R_+$
and
$\Phi_2\circ b_{2,B} = s_A\circ h^{-1}$
to prove  that 
$(X_A,\sigma_A)$ and 
$(X_B,\sigma_B)$ are one-sided flow equivalent.
\end{proof}
Conversely we have
\begin{proposition}
If $(X_A,\sigma_A)$ and 
$(X_B,\sigma_B)$ are one-sided flow equivalent,
then they are continuously orbit equivalent.
\end{proposition}
\begin{proof}
Suppose that 
$(X_A,\sigma_A)$ and 
$(X_B,\sigma_B)$ are one-sided flow equivalent.
Take suspension triplets $(l_1,k_1,b_1)$ for $(X_A,\sigma_A)$
and $(l_2,k_2,b_2)$ for $(X_B,\sigma_B)$,
a homeomorphism
$h:X_A \rightarrow X_B$,
and continuous maps
$\Phi_1:S_{A,b_1}^{l_1,k_1}\rightarrow S_{B}^{1}$,
$\Phi_2:S_{B,b_2}^{l_2,k_2}\rightarrow S_{A}^{1}$
satisfying
the equalities
\eqref{eq:crosssectionB} and \eqref{eq:crosssectionA}.
For 
$(x,r)\in X_{A,b_1}^{\R}$
with $r \ge l_1(x)$,
we have
$[x,r] = [\sigma_A(x),r-c_1(x)]$ in $S_{A,b_1}^{l_1,k_1}$.
It follows that
\begin{align*}
[x,l_1(x)] & =[x,(l_1(x)-b_1(x)) + b_1(x)], \\ 
[\sigma_A(x),l_1(x)-c_1(x) ]
& = [\sigma_A(x),(k_1(x) -b_1(\sigma_A(x))) + b_1(\sigma_A(x))].
\end{align*} 
As $l_1(x)-b_1(x) \ge 0$ and $k_1(x) -b_1(\sigma_A(x)) \ge 0$,
we have
\begin{align*}
[x,l_1(x)] 
& = \phi_{A,l_1(x)-b_1(x)}([x,b_1(x)])
  = \phi_{A,l_1(x)-b_1(x)}(b_{1,A}(x)), \\ 
[\sigma_A(x),l_1(x)-c_1(x) ]
& = \phi_{A,k_1(x)-b_1(\sigma_A(x))}([\sigma_A(x),b_1(\sigma_A(x))]) \\
& = \phi_{A,k_1(x)-b_1(\sigma_A(x))}(b_{1,A}(\sigma_A(x))).
\end{align*} 
Hence we have
$
 \phi_{A,l_1(x)-b_1(x)}(b_{1,A}(x))
= \phi_{A,k_1(x)-b_1(\sigma_A(x))}(b_{1,A}(\sigma_A(x)))
$ 
so that
\begin{equation*}
\Phi_1( \phi_{A,l_1(x)-b_1(x)}(b_{1,A}(x)))
=\Phi_1( \phi_{A,k_1(x)-b_1(\sigma_A(x))}(b_{1,A}(\sigma_A(x)))).
\end{equation*} 
By \eqref{eq:crosssectionB} and \eqref{eq:crosssectionA},
we have
\begin{equation*}
  \phi_{B,l_1(x)-b_1(x)}(s_B(h(x)))
= \phi_{B,k_1(x)-b_1(\sigma_A(x))}(s_B(h(\sigma_A(x)))).
\end{equation*} 
It follows that
\begin{equation*}
  [h(x),l_1(x)-b_1(x)]
= [h(\sigma_A(x)),k_1(x)-b_1(\sigma_A(x))] \quad 
\text{ in } S_{B}^{1}.
\end{equation*}
Put
$l_1'(x) =l_1(x)-b_1(x)$
and
$k_1'(x) = k_1(x)-b_1(\sigma_A(x))$.
They are valued in nonnegative integers.
Since
\begin{align*}
  [h(x),l_1(x)-b_1(x)]
& =[\sigma_B^{l_1'(x)}(h(x)),0]= s_B(\sigma_B^{l_1'(x)}(h(x))),\\
  [h(\sigma_A(x)),k_1(x)-b_1(\sigma_A(x))]
& =[\sigma_B^{k_1'(x)}(h(\sigma_A(x))),0]= s_B(\sigma_B^{k_1'(x)}(h(\sigma_A(x)))),
\end{align*}
and 
the standard base map $s_B:X_B \rightarrow S_{B}^{1}$
is injective,
we have
$\sigma_B^{l_1'(x)}(h(x))=\sigma_B^{k_1'(x)}(h(\sigma_A(x)))$
for $x \in X_A$.
We similarly have  continuous maps
$l_2', k_2' \in C(X_B,\Zp)$
such that 
$\sigma_A^{l_2'(y)}(h^{-1}(y))=\sigma_A^{k_2'(y)}(h^{-1}(\sigma_B(y)))$
for $y \in X_B$.
Consequently 
$(X_A,\sigma_A)$ and 
$(X_B,\sigma_B)$ are continuously orbit equivalent.
\end{proof}
Therefore we may conclude the following theorem.
\begin{theorem}\label{thm:fecoe}
One-sided topological Markov shifts 
$(X_A, \sigma_A)$ and $(X_B,\sigma_B)$ 
are one-sided flow equivalent
if and only if they 
are continuously orbit equivalent.
\end{theorem}
It is well-known that 
two-sided topological Markov shifts 
$(\bar{X}_A,\bar{\sigma}_A)$ 
and 
$(\bar{X}_B,\bar{\sigma}_B)$
are flow equivalent if and only if 
$\Z^N/(\id - A){\Z}^N$ is isomorphic to
$\Z^M/(\id - B){\Z}^M$
as abelian groups
and
$\det(\id - A) =\det(\id - B)$,
where $N, M$ are the sizes of the matrices $A, B$
respectively
(\cite{BF}, \cite{Franks}, \cite{PS}).
By using \cite[Theorem 3.6]{MMETDS},
we see the following characterization of one-sided flow equivalence
which is a corollary of the above theorem.
\begin{corollary}\label{cor:flmatrix}  
One-sided topological Markov shifts
$(X_A, \sigma_A)$ and $(X_B,\sigma_B)$ 
are one-sided flow equivalent
if and only if
there exists an isomorphism
$\varPhi: \Z^N/(\id - A){\Z}^N 
\rightarrow 
\Z^M/(\id - B){\Z}^M$
of abelian groups
such that 
$\varPhi([u_A]) = [u_B]$
and
$\det(\id - A) =\det(\id - B)$,
where
$[u_A]$ (resp. $[u_B]$)  
is the class of the vector 
$u_A =[1,\dots,1]$ in $\Z^N/(\id - A){\Z}^N$
(resp. 
$u_B =[1,\dots,1]$ in $\Z^M/(\id - B){\Z}^M$).
\end{corollary}
We note that  
the statement of the following proposition 
is more general than that of Proposition \ref{prop:3.2}.
\begin{proposition}
Suppose that $(X_A,\sigma_A)$ and $(X_B,\sigma_B)$
are continuously orbit equivalent via a homeomorphism 
$h:X_A\rightarrow X_B$.
For 
$f \in C(X_B,\Zp)$ and 
$g \in C(X_A,\Zp)$ such that  
$[f] \in H_+^B$ and  $[g] \in H_+^A$ 
are order units 
of $(H^B,H^B_+)$ 
and 
of $(H^A,H^A_+)$
respectively,
there exist
$l_f, k_f \in C(X_A,\Zp)$ 
and
$l_g, k_g \in C(X_B,\Zp)$ such that 
$(l_f, k_f, 0)$ and 
$(l_g, k_g, 0)$ are suspension triplets for
$(X_A, \sigma_A)$ and for $(X_B,\sigma_B)$ 
respectively
and
continuous maps
$\Phi_f:S_{A}^{l_f,k_f}\rightarrow S_B^f$
and
$\Phi_g:S_{B}^{l_g,k_g}\rightarrow S_A^g$
such that
\begin{alignat*}{3}
\Phi_f\circ \phi_{A,t} & = \phi_{B,t}\circ \Phi_f
\quad\text{for } t \in \R_+,
\qquad
\Phi_f\circ s_A & = & s_B \circ h, 
\\
\Phi_g\circ \phi_{B,t} & = \phi_{A,t}\circ \Phi_g
\quad\text{for } t \in \R_+,
\qquad
\Phi_g\circ s_B & = & s_A \circ h^{-1}.
\end{alignat*}
\end{proposition}
\begin{proof}
Put
\begin{equation*}
l_f(x)  = f^{l_1(x)}(h(x)),
\qquad
k_f(x)  = f^{k_1(x)}(h(\sigma_A(x))), 
 \qquad x \in X_A
\end{equation*}
so that $\Psi_h(f)(x) = l_f(x) - k_f(x)$.
Since $[f]\in H^B_+$ is an order unit and 
$\Psi_h:H^B\rightarrow H^A$ preserves the orders,
the class
$[\Psi_h(f)]$ gives rise to an order unit of $(H^A,H^A_+)$.
As $l_f - k_f =\Psi_h(f)$,
we see that $(l_f,k_f,0)$ is a suspension triplet for $(X_A,\sigma_A)$
so that we may consider the one-sided suspensions
$(S_A^{l_f,k_f},\phi_A)$ and
$(S_B^f,\phi_B)$.
We define the map
$\Phi_f:X_{A,0}^\R \rightarrow X_{B,0}^\R$
by $\Phi_f((x,r)) = (h(x),r)$.
For $r \ge l_f(x)$, we have 
$$
(h(x),r) 
\underset{f,0}{\sim}(\sigma_B(h(x)), r-f(h(x)))
\underset{f,0}{\sim}\cdots
\underset{f,0}{\sim}(\sigma_B^{l_1(x)}(h(x)), r-f^{l_1(x)}(h(x)))
$$
and
similarly
$$
(h(\sigma_A(x)),r-\Psi_h(f)(x))
\underset{f,0}{\sim}(\sigma_B^{k_1(x)}(h(\sigma_A(x))), 
r-\Psi_h(f)(x)-f^{k_1(x)}(h(\sigma_A(x)))).
$$
As
the equalities 
$\sigma_B^{l_1(x)}(h(x))=\sigma_B^{k_1(x)}(h(\sigma_A(x)))$
and
$f^{l_1(x)}(h(x))= 
\Psi_h(f)(x) + f^{k_1(x)}(h(\sigma_A(x))
$
hold,
we have
$$
\Phi_f((x,r)) 
\underset{f,0}{\sim}
\Phi_f((\sigma_A(x),r-\Psi_h(f)(x)). 
$$
Hence
$\Phi_f:X_{A,0}^\R \rightarrow X_{B,0}^\R$
induces 
a continuous map
$S_A^{l_f,k_f} \rightarrow S_B^f$
which is still denoted by 
$\Phi_f$.
It is easy to see that the map satisfies the desired properties.
We similarly have a desired map
$\Phi_g:S_B^{l_g,k_g} \rightarrow S_A^g$.
\end{proof}

We give an example of one-sided flow equivalent topological Markov shifts.
Let 
$A =
\left[\begin{smallmatrix}
1 & 1\\
1 & 1
\end{smallmatrix}\right],
\,
B =
\left[\begin{smallmatrix}
1 & 1\\
1 & 0
\end{smallmatrix}\right].
$
The one-sided topological Markov shift
$(X_A,\sigma_A)$ is called the full $2$-shift,
and the other one 
$(X_B,\sigma_B)$ is called the golden mean shift
whose shift space $X_B$ consists of
sequences $(y_n)_{n \in \N}$ of $1,2 $ such that the word $(2,2)$ 
is forbidden (cf. \cite{LM}).
They are continuously orbit equivalent as in \cite[Section 5]{MaPacific}   
through the homeomorphism
$h: X_A \longrightarrow X_B$
defined 
by substituting the word $(2,1)$ for the symbol $2$ 
from the leftmost of a sequence $(x_n)_{n\in \N}$
in order,
so that they are one-sided flow equivalent.
Put for $i=1,2$
\begin{equation*} 
U_{A,i} 
=  \{ (x_n)_{n \in \N} \in X_A \mid x_1 =i\},\qquad
U_{B,i}  
 =  
\{ (y_n)_{n \in \N}  \in X_B \mid y_1 =i\}.
\end{equation*}
By setting
\begin{equation*}
\begin{cases}
k_1(x) =0, \, l_1(x) =1 & \text{ for } x \in  U_{A,1},\\
k_1(x) =0, \, l_1(x) =2 & \text{ for } x \in  U_{A,2},
\end{cases}
\qquad
\begin{cases}
k_2(y) =0, \, l_2(y) =1 & \text{ for } y \in  U_{B,1},\\
k_2(y) =1, \, l_2(y) =1 & \text{ for } y \in  U_{B,2},
\end{cases}
\end{equation*}
the continuous functions
$k_1,l_1: X_A\rightarrow\Zp$ 
and 
$k_2,l_2:X_B\rightarrow \Zp$
satisfy \eqref{eq:orbiteqx}
and \eqref{eq:orbiteqy}, respectively.
By Proposition \ref{prop:3.2},
the suspension flows 
$(S_A^{l_1}, \phi_A)$ and $(S_B^1, \phi_B)$,
and similarly
$(S_B^{l_2,k_2}, \phi_B)$ and $(S_A^1, \phi_A)$
give rise to flow equivalence 
between 
$(X_A,\sigma_A)$ 
and
$(X_B,\sigma_B)$.


\section{One-sided suspensions and  zeta functions}
The zeta function $\zeta_{\phi}(s)$
of a flow $\phi_t:S \rightarrow S$ 
on a compact metric space with at most countably many closed orbits 
is defined by the formula \eqref{eqn:zetaflow}.
In the first half of this section, we will show that the zeta function 
$\zeta_{\phi_A}(s)$ of the flow $\phi_A$ of the one-sided suspension 
$S_{A,b}^{l,k}$ of $(X_A,\sigma_A)$ is given by the dynamical zeta function
$\zeta_{A,c}(s)$ with potential function $c = l-k$.
We first provide a lemma. 
\begin{lemma}\label{lem:period}
Let $(S_{A,b}^{l,k},\phi_A)$ be a  one-sided suspension of $(X_A,\sigma_A)$.
Then there exists a bijective correspondence between primitive periodic orbits
$\tau \in P_{orb}(S_{A,b}^{l,k},\phi_A)$ 
and periodic orbits 
$\gamma_\tau \in P_{orb}(X_A)$ of $(X_A,\sigma_A)$ such that
\begin{equation*}
\ell(\tau) = \beta_{\gamma_\tau}(c)
\end{equation*}
where 
$\ell(\tau)$ is the primitive length of the periodic orbit 
$\tau$
defined by
 $\ell(\tau) = \min\{t\in \R_+ \mid \phi_{A,t}(u) = u\}$ 
for any point $u \in \tau$
and
$\beta_{\gamma_\tau}(c)=\sum_{i=0}^{p-1}c(\sigma_A^i(x))$
for  $c = l-k$ and 
$\gamma_\tau =\{x,\sigma_A(x), \dots,\sigma_A^{p-1}(x)\}.
$
\end{lemma}
\begin{proof}
For an arbitrary point $(x,r)$ in 
a primitive periodic orbit
$\tau \in P_{orb}(S_{A,b}^{l,k},\phi_A)$,
one sees that
$(x,r) \underset{l,k}{\sim}
\Phi_{A,\ell(\tau)}(x,r)$.
Since  
$\Phi_{A,\ell(\tau)}(x,r)= (x,r+\ell(\tau))$,
there exists a number $p \in \Zp$ such that 
\begin{equation*}
\ell(\tau) = 
(l(x) -r) 
+ 
\{l(\sigma_A(x)) - k(x)\} 
+ 
\cdots 
+
\{ l(\sigma_A^{p-1}(x)) - k(\sigma_A^{p-2}(x))\} 
+
\{r- k(\sigma_A^{p-1}(x))\} 
\end{equation*}
and 
$\sigma_A^p(x) = x$,
so that 
$\ell(\tau)=\sum_{i=0}^{p-1}c(\sigma_A^i(x))$.

Conversely, for a periodic orbit 
$\gamma =\{x,\sigma_A(x), \dots, \sigma_A^{p-1}(x)\} \in P_{orb}(X_A)$,
we have
$(x,r) \underset{l,k}{\sim}
\Phi_{A,\ell(\tau)}(x,r)$
for any $r \in \R$ with $b(x) \le r \le l(x)$.
\end{proof}
Let us denote by 
$\Per_n({X}_A)$
the set
$
\{ {x} \in {X}_A\mid {\sigma}_A^n({x}) = {x}\}
$
of $n$-periodic points.
For a H{\"{o}}lder continuous function
$f$ on $X_A$,
the dynamical zeta function
$\zeta_{A,f}(s)$
is defined by 
\begin{equation}
\zeta_{A,f}(s)
= \exp\{
\sum_{n=1}^\infty\frac{1}{n}\sum_{{x}\in \Per_n({X}_A)}
\exp( -s \sum_{k=0}^{n-1}{f}({\sigma}_A^k({x}))) \} \qquad
(\text{see } \cite{Ruelle1978},\cite{Ruelle2002}, \cite{PP}, \text{ etc}.)
\label{eq:fdynamiczeta}
\end{equation}
where the right hand side makes sense for a complex number
$s \in \C$ with 
${\Re}(s)>h$ for some positive constant $h>0$.
The function 
$\zeta_{A,f}(s)$
is called 
the dynamical zeta function on $X_A$
with potential function $f$. 
We may especially define the zeta function
$\zeta_{A,c}(s)$ 
for an integer valued continuous function
$c$ on $X_A$ whose class $[c]$ in $H^A$ is an order unit of 
$(H^A,H^A_+)$.
By a routine argument as in \cite[p.100]{PP}
with Lemma \ref{lem:period},
we have 
\begin{proposition}\label{prop:zetaformula} 
The zeta function 
\eqref{eqn:zetaflow}
of the flow of the one-sided suspension 
$(S_{A,b}^{l,k},\phi_{A})$
of $(X_A,\sigma_A)$
is given by the dynamical zeta function 
$\zeta_{A,c}(s)$ with potential function $c=l-k$ such as 
\begin{equation}
\zeta_{A,c}(s)
= \exp\{
\sum_{n=1}^\infty\frac{1}{n}\sum_{{x}\in \Per_n({X}_A)}
\exp( -s \sum_{k=0}^{n-1}{c}({\sigma}_A^k({x}))) \}. 
\label{eq:dynamiczeta}
\end{equation}
\end{proposition}
\begin{remark}
The class 
$[c]$ 
of the function $c$ in $H^A$ is an order unit
of the ordered cohomology group
$(H^A,H^A_+)$,
so that there exists $N_c \in \N$ such that 
$N_c [c] -1 \in H^A_+$.
Hence $N_c -1 = f + g \circ \sigma_A -g$
for some $f \in C(X_A,\Zp)$ and $g \in C(X_A,\Z)$.
For a periodic point $x \in \Per_n(X_A)$ 
with 
$\sigma_A^n(x) = x$, 
we have
\begin{equation*}
\sum_{i=0}^{n-1}
\{ N_c c(\sigma_A^i(x)) -1(\sigma_A^i(x))\}
= 
\sum_{i=0}^{n-1}
f(\sigma_A^i(x)) 
\ge 0
\end{equation*}
and hence  
$
\sum_{i=0}^{n-1}c(\sigma_A^i(x)) \ge \frac{n}{N_c}.
$
The ordinary zeta function
$\zeta_A(t) $ of $(X_A,\sigma_A)$ 
is written as 
$\zeta_A(t) = \exp\{
\sum_{n=1}^\infty\frac{1}{n}
|\Per_n({X}_A)|t^n \} 
$
where
$|\Per_n({X}_A)|$
denotes the cardinarity of $n$-periodic points
$\Per_n({X}_A)$,
so that 
$\zeta_A(t) = \zeta_{A,1}(s)$
for $c=1$ 
where
$t = e^{-s}$.
As the function 
$\zeta_A(t)$
is analytic in $t\in \C$ with $|t|<\frac{1}{r_A}$,
where $r_A$ is the maximum eigenvalue of the matrix $A$,
we see that
so is $\zeta_{A,1}(s)$
in $s \in \C$ with $\Re(s) > \log r_A$.
Similarly, for the function $c = l-k$, 
we have
\begin{equation*}
|\sum_{{x}\in \Per_n({X}_A)}
\exp( -s \sum_{k=0}^{n-1}{c}({\sigma}_A^k({x})))| 
\le 
\sum_{{x}\in \Per_n({X}_A)}
\{\exp( - \Re(s))\}^{\frac{n}{N_c}} 
\end{equation*}
so that 
$\zeta_{A,c}(s)$
is analytic at least in $s \in \C$ with 
$\Re(s) > N_c \log r_A.$ 
We actually know that 
$\zeta_{A,c}(s)$
is analytic in the half plane 
$\Re(s) > h_{top}(S_{A,b}^{l,k},\phi_A)$
the topological entropy of the flow of the suspension
 $(S_{A,b}^{l,k},\phi_A)$
 as seen in \cite[Theorem 2.7]{Baladi1998}
(cf. \cite{Haydn}, \cite{Pollicott1986}, \cite{Ruelle1987}). 
\end{remark}

In the second half of this section,
we will study 
some relationship
between zeta functions of  one-sided flow equivalent
topological Markov shifts. 
Suppose that
$(X_A,\sigma_A)$ and $(X_B,\sigma_B)$
are continuously orbit equivalent
via a homeomorphism $h: X_A\rightarrow X_B$
with continuous functions 
$k_1,l_1: X_A\rightarrow\Zp$ 
and 
$k_2,l_2:X_B\rightarrow \Zp$
satisfying \eqref{eq:orbiteqx}
and \eqref{eq:orbiteqy}, respectively.
The  functions 
$c_1(x) = l_1(x) - k_1(x), x \in X_A$
and
$c_2(y) = l_2(y) - k_2(y), y \in X_B$
satisfy
$\Psi_h(1) =c_1$
and
$\Psi_{h^{-1}}(1) = c_2$.
In \cite{MMETDS}, it has been shown that 
the homeomorphism $h: X_A\rightarrow X_B$
induces a bijective correspondence
$\xi_h: P_{orb}({X}_A) \rightarrow P_{orb}({X}_B)$
between their periodic orbits
such that the functions
$c_1, c_2$ measure the difference 
of the length of periods between
$\gamma \in P_{orb}({X}_A)$
and 
$\xi_h(\gamma) \in P_{orb}({X}_B)$.
As a result,
the ordinary zeta functions
$\zeta_A(t), \zeta_B(t)$
heve been proved to be written in terms of  
dynamical zeta functions
 in the following way.
\begin{proposition}[\cite{MMETDS}] \label{prop:dzeta}
$\zeta_A(t) = \zeta_{B,c_2}(s)$
and
$\zeta_B(t) = \zeta_{A,c_1}(s)$
for $t=e^{-s}$.
\end{proposition}
The above formulae imply the formulae
\begin{equation}
\zeta_{A,1}(s) = \zeta_{B,c_2}(s), \qquad
\zeta_{B,1}(s) = \zeta_{A,c_1}(s).
\end{equation}
Assume  that 
$(X_A, \sigma_A)$ and $(X_B,\sigma_B)$
are one-sided flow equivalent 
via a homeomorphism
$h:X_A\rightarrow X_B$.
They are continuously orbit equivalent,
so that 
the maps
$\Psi_h: C(X_B,\Z) \rightarrow C(X_A,\Z)$
and  
$\Psi_{h^{-1}}: C(X_A,\Z) \rightarrow C(X_B,\Z)$
are defined
by \eqref{eq:Psihfx}.
They are independent of the choice of the functions
$k_1,l_1: X_A\rightarrow\Zp$ 
and 
$k_2,l_2:X_B\rightarrow \Zp$
satisfying \eqref{eq:orbiteqx}
and \eqref{eq:orbiteqy}, respectively
(\cite[Lemma 4.2]{MMETDS}).
We provide a lemma.
\begin{lemma}\label{lem:4.3}
For $m \in \Zp$ and $f \in C(X_B,\Z)$, $g \in C(X_A,\Z)$, 
we have
\begin{enumerate}
\renewcommand{\theenumi}{\roman{enumi}}
\renewcommand{\labelenumi}{\textup{(\theenumi)}}
\item
$\Psi_h(f)^m(x) 
 = f^{l_1^m(x)}(h(x)) - f^{k_1^m(x)}(h(\sigma_A^m(x)))
 $
 for $x \in X_A$,
 so that 
 $$
g^m(x) 
 = \Psi_{h^{-1}}(g)^{l_1^m(x)}(h(x)) 
- \Psi_{h^{-1}}(g)^{k_1^m(x)}(h(\sigma_A^m(x))).
$$
\item
$
\Psi_{h^{-1}}(g)^m(y) 
 = g^{l_2^m(y)}({h^{-1}}(y)) - g^{k_2^m(y)}({h^{-1}}(\sigma_B^m(y)))
 $
 for $y \in X_B$,
 so that
 $$
f^m(y) 
 = \Psi_{h}(f)^{l_2^m(y)}({h^{-1}}(y)) 
- \Psi_{h}(f)^{k_2^m(y)}({h^{-1}}(\sigma_B^m(y))).
$$
\end{enumerate}
\end{lemma}
\begin{proof}
(i)
As in \cite[Lemma 4.3]{MMETDS},
the identity
\begin{align*}
& \sum_{i=0}^{m-1} \{
\sum_{i'=0}^{l_1(\sigma_A^i(x))-1} f(\sigma_B^{i'}(h(\sigma_A^i(x))))
-
\sum_{j'=0}^{k_1(\sigma_A^i(x))-1} f(\sigma_B^{j'}(h(\sigma_A^{i+1}(x)))) \}\\
=
&
\sum_{i'=0}^{l_1^m (x)-1} f(\sigma_B^{i'}(h(x)))
-
\sum_{j'=0}^{k_1^m (x)-1} f(\sigma_B^{j'}(h(\sigma_A^{m}(x))))
\end{align*}
holds so that we see 
\begin{equation*}
\sum_{i=0}^{m-1}\Psi_h(f)(\sigma_A^i(x)) 
= f^{l_1^m(x)}(h(x)) - f^{k_1^m(x)}(h(\sigma_A^m(x))).
\end{equation*}
As $\Psi_{h^{-1}} = {(\Psi_h)}^{-1} $(\cite[Proposition 4.5]{MMETDS}), 
the desired identities hold.
(ii) is similarly shown.
\end{proof}

We generalize Proposition \ref{prop:dzeta}
such as the following theorem.
\begin{theorem} \label{thm:fezeta}
Suppose that
$(X_A, \sigma_A)$ and $(X_B,\sigma_B)$
are one-sided flow equivalent 
via a homeomorphism
$h:X_A\rightarrow X_B$.
Then  for 
$f \in C(X_B, \Z), 
g \in C(X_A, \Z)$
such that
the classes 
$[f], [g]$ are order units of the ordered cohomology groups
$(H^B,H^B_+),  (H^A,H^A_+)$,
respectively.
Then  
we have
\begin{equation*}
\zeta_{A,g}(s) = \zeta_{B,\Psi_{h^{-1}}(g)}(s),
\qquad
\zeta_{B,f}(s)=\zeta_{A,\Psi_h(f)}(s).
\end{equation*}
\end{theorem}
\begin{proof}
We may assume that 
$(X_A,\sigma_A)$ and $(X_B,\sigma_B)$
are continuously orbit equivalent
via a homeomorphism $h: X_A\rightarrow X_B$.
For $ f \in C(X_B,\Z)$ such that 
the class $[f]$ 
is an order unit of $(H^A,H^A_+)$,
we will prove the equality
$
 \zeta_{B,f}(s)
 =\zeta_{A,\Psi_h(f)}(s).
$
A routine argument as in \cite[p. 100]{PP} shows 
\begin{equation*}
\zeta_{A,\Psi_h(f)}(s) 
= \prod_{\gamma \in P_{orb}({X}_A)}
(1 - t^{\beta_\gamma(\Psi_h(f))})^{-1} 
\quad
\text{ where } 
\quad t =e^{-s}
\end{equation*}
and
$\beta_\gamma(\Psi_h(f))
=\sum_{i=0}^{p-1}\Psi_h(f)(\sigma_A^i(x))
$
for a periodic orbit 
$\gamma = \{x,\sigma_A(x),\dots,\sigma_A(^{p-1}(x)\}\in P_{orb}({X}_A)$. 
We see the following formula by Lemma \ref{lem:4.3},
\begin{equation*}
\sum_{i=0}^{p-1}\Psi_h(f)(\sigma_A^i(x)) 
 = f^{l_1^p(x)}(h(x)) - f^{k_1^p(x)}(h(\sigma_A^p(x))),
\quad x \in X_A.
\end{equation*}
As 
$\gamma=\{ x,\sigma_A(x), \dots,\sigma_A^{p-1}(x)\}
 \in P_{orb}({X}_A)
$
and  $\sigma_A^p(x) = x$,
we have
\begin{equation*}
\sum_{i=0}^{p-1}\Psi_h(f)(\sigma_A^i(x)) 
 = f^{l_1^p(x)}(h(x)) - f^{k_1^p(x)}(h(x)).
\end{equation*}
Since we may identify 
the periodic orbits 
$P_{orb}({X}_A)$
of the one-sided topological Markov shift 
$(X_A,\sigma_A)$
with 
the periodic orbits 
$P_{orb}(\bar{X}_A)$
of the two-sided topological Markov shift 
$(\bar{X}_A, \bar{\sigma}_A)$,
an argument in \cite[Section 6]{MMETDS}
shows that there exists a bijective correspondence
$\xi_h: P_{orb}({X}_A)\rightarrow P_{orb}({X}_B)$.
By \cite[Lemma 6.5]{MMETDS},
we see
that $\xi_h(\gamma)$ has its period
$l_1^p(x)-k_1^p(x)$
and hence
$\xi_h(\gamma) 
= \{ \sigma_B^i(h(x)) \mid {k_1^p(x)} \le i \le {l_1^p(x)}-1 \}$
so that
$$
 \sum_{i=k_1^p(x)}^{l_1^p(x)-1}f(\sigma_B^i(h(x))) 
 = \beta_{\xi_h(\gamma)}(f).
$$
We then have 
\begin{align*}
\beta_\gamma(\Psi_h(f))
& = \sum_{i=0}^{p-1}\Psi_h(f)(\sigma_A^i(x)) 
 = \sum_{i=0}^{l_1^p(x)-1}f(\sigma_B^i(h(x))) -
    \sum_{i=0}^{k_1^p(x)-1}f(\sigma_B^i(h(x))) \\
& = \sum_{i=k_1^p(x)}^{l_1^p(x)-1}f(\sigma_B^i(h(x))) 
= \beta_{\xi_h(\gamma)}(f).
\end{align*}
Since
$\xi_h:P_{orb}({X}_A)\rightarrow P_{orb}({X}_B)
$
is bijective,
one sees that 
\begin{equation*}
\zeta_{A,\Psi_h(f)}(s) 
 = \prod_{\eta \in P_{orb}(X_B)}
(1 - t^{\beta_\eta(f)})^{-1}
= \zeta_{B,f}(s).
\end{equation*}
The other equality
$ 
\zeta_{B,\Psi_{h^{-1}}(g)}(s) = \zeta_{A,g}(s)
$
is similarly shown.
\end{proof}

\begin{corollary}
The set $Z(X_A,\sigma_A)$
of dynamical zeta functions of $(X_A,\sigma_A)$
whose potential functions are order units of the ordered cohomology
group $(H^A,H^A_+)$
is invariant under one-sided flow equivalence.
\end{corollary}



\end{document}